\documentclass{amsart}
\usepackage{graphicx}
\usepackage{amsfonts}
\usepackage{float}
\newtheorem{tm}{Theorem}

\newtheorem{defi}{Definition}

\newtheorem{rem}{Remark}
\newtheorem{rems}{Remarks}
\newtheorem{lm}{Lemma}
\newtheorem{ex}{Example}
\newtheorem{cor}{Corollary}
\newtheorem{prop}{Proposition}
\newtheorem{nota}{Notation}

\newtheorem{prob}{Problem}

\begin{document}

\title{On universal sign patterns}

\author{Vladimir Petrov Kostov}

\address{Universit\'e C\^ote d’Azur, CNRS, LJAD, France}
\email{vladimir.kostov@unice.fr}

\begin{abstract}
  We consider polynomials $Q:=\sum _{j=0}^da_jx^j$, $a_j\in \mathbb{R}^*$,
  with all roots real. When the {\em sign pattern}
  $\sigma (Q):=({\rm sgn}(a_d),{\rm sgn}(a_{d-1})$, $\ldots$, ${\rm sgn}(a_0))$
  has $\tilde{c}$ sign changes,
  the polynomial $Q$ has $\tilde{c}$ positive and $d-\tilde{c}$ negative roots.
  We suppose the moduli of these roots distinct. The {\em order} of these
  moduli is defined when in their string as points of the positive half-axis
  one marks
  the places of the moduli of negative roots. A sign pattern $\sigma^0$
  is {\em universal} when
  for any possible order of the moduli there exists a polynomial $Q$
  with $\sigma (Q)=\sigma^0$ and with this order of the moduli of its roots.
  We show that when the polynomial $P_{m,n}:=(x-1)^m(x+1)^n$ has no vanishing
  coefficients, the sign pattern $\sigma (P_{m,n})$ is universal. We also 
  study the question when $P_{m,n}$ can have vanishing coefficients. 
  
  {\bf Key words:} real polynomial in one variable; hyperbolic polynomial;
  sign pattern; Descartes' rule of signs\\

{\bf AMS classification:} 26C10; 30C15
\end{abstract}
\maketitle

\section{Introduction}

The present paper deals with univariate real polynomials
$Q:=\sum _{j=0}^da_jx^j$, $a_j\in \mathbb{R}$, $a_d\neq 0$. The classical
Descartes' rule of signs states that the number $\ell_+$ of positive roots
of such a polynomial (counted with multiplicity) is majorized by the number
$\tilde{c}$ of sign changes in the sequence of its coefficients and the
difference $\tilde{c}-\ell_+$ is even, see \cite{Ca}, \cite{Cu}, \cite{DG},
\cite{Des}, \cite{FoNoSh}, \cite{Fo}, \cite{Ga}, \cite{J}, \cite{La} or
\cite{Mes}. (When
defining the number $\tilde{c}$ one has to erase the zero coefficients.) 
We are mainly concerned with {\em hyperbolic} polynomials, i.~e. ones with all
roots real. If all coefficients of a hyperbolic polynomial are
non-vanishing, then $\ell_+=\tilde{c}$ and $\ell_-=\tilde{p}$, where $\ell_-$ is
the number of negative roots and $\tilde{p}$ is the number of sign
preservations; clearly $\ell_++\ell_-=\tilde{c}+\tilde{p}=d$.

We are interested in a problem closely related to Descartes' rule of signs.

\begin{defi}
  {\rm A {\em sign pattern} of length $d+1$ is a string of $d+1$ signs
    $+$ and/or $-$. If all coefficients $a_j$ are non-zero, then we say
    that the polynomial $Q$ defines the sign pattern
    $\sigma (Q):=({\rm sgn}(a_d),{\rm sgn}(a_{d-1})$,
    $\ldots$, ${\rm sgn}(a_0))$.
    If there is at least one vanishing coefficient, then we say that
    $\sigma (Q)$ is a {\em generalized sign pattern} and its components
    can be equal to $+$, $-$ or~$0$. For a given sign pattern, we define
    its corresponding {\em change-preservation pattern} by writing the letter
    $c$
    when there are two consecutive signs $(+,-)$ or $(-,+)$ and the letter
    $p$ if
    there are signs $(+,+)$ or $(-,-)$. Thus the change-preservation pattern
    corresponding to the sign pattern $(+,+,-,-,-,+,-,+)$ is $pcppccc$.
    We consider mainly monic polynomials, so the correspondence between sign
    and change-preservation patterns is bijective.}
\end{defi}

One can consider the moduli of the roots of $Q$. We limit ourselves to the
case when all these moduli are distinct. The definition of the
{\em order} of these moduli should be clear from the following example.

\begin{ex}
  {\rm Suppose that $d=8$ and the polynomial $Q$ has positive roots
    $\alpha_1<\alpha_2<\alpha_3$ and negative roots $-\gamma _j$, $j=1$,
    $\ldots$, $5$, such that}

  $$\gamma_1<\alpha_1<\gamma _2<\gamma_3<\gamma_4<\alpha_2<\gamma_5<\alpha_3~.$$
  {\rm Then we say that the order of moduli defined by the roots of $Q$ is
    $NPNNNPNP$, i.~e. the letters $P$ and $N$ indicate the relative positions
    of the moduli of positive and negative roots.}
\end{ex}

\begin{defi}\label{deficompat}
  {\rm (1) A couple (sign pattern, order of moduli) is {\em compatible} if the
    number of sign changes (resp. sign preservations) of the sign pattern is
    equal to the number of letters $P$ (resp. $N$) of the order of moduli.
    In what follows we consider only compatible couples
    and we write ``couples'' for short.
    \vspace{1mm}
    
    (2) A couple is {\em realizable} if there
    exists a hyperbolic polynomial which defines the sign pattern of the couple
    and the moduli of whose roots define the given order. We can also say that
    the given sign pattern is realizable with the given order of moduli or
  vice versa.}
\end{defi}

In the present paper we consider the following problem:

\begin{prob}\label{prob1}
  For a given degree $d$, which couples are realizable?
  \end{prob}

The problem is settled for $d\leq 6$, see \cite{KoSo} and \cite{GaKoTaMC}.
We remind now some results towards the resolution of Problem~\ref{prob1}.

\begin{defi}\label{defican}
  {\rm Given a sign pattern $\sigma^*$ one defines its corresponding
    {\em canonical} order of moduli as follows: one reads $\sigma^*$
    from the right and to each sign change (resp. sign preservation) one puts
    into correspondence the letter $P$ (resp. $N$). Example: the canonical
    order corresponding to the sign pattern $(+,-,-,-,+,+,-,+,-,-,+)$ is
  $PNPPPNPNNP$.}
\end{defi}

\begin{rems}\label{remscanrig}
  {\rm (1) Each couple (sign pattern, corresponding canonical order)
    is realizable,
see~\cite[Proposition~1]{KoSe}. However there exist sign patterns
(called also {\em canonical}, see~\cite{KoRM}) which are realizable
{\em only} with
their corresponding canonical orders. They are characterized by the property
not to contain four consecutive signs $(+,+,-,-)$, $(-,-,+,+)$,
$(+,-,-,+)$ or $(-,+,+,-)$, so the corresponding change-preservation patterns
have no triple $cpc$ or $pcp$, i.~e. no isolated sign change and no
isolated sign preservation except at the two ends.
\vspace{1mm}

(2) There exist also orders of moduli (called {\em rigid}) which are realizable
with a single sign pattern. Except the {\em trivial case} when all roots are
of the same sign and the change-preservation pattern is $cc\ldots c$ or
$pp\ldots p$, rigid are the orders $\ldots NPNPNP$ and $\ldots PNPNPN$,
realizable with the sign patterns $(+,-,-,+,+,-,-,+,+,\ldots )$ and
$(+,+,-,-,+,+,-,-,+,+,\ldots )$ respectively, see~\cite{KorigMO}.}
\end{rems}

The present paper explores the inverse phenomenon, i.~e. 
sign patterns realizable with all compatible orders of moduli.

\begin{rem}
  {\rm Excluding the trivial case, orders of moduli realizable
    with all compatible sign patterns do not exist. Indeed, if such an order
    $\Omega$ exists (with $\tilde{c}>0$ letters $P$ and $\tilde{p}>0$
    letters $N$), then it should be realizable with each of the two
    change-preservation patterns, $cc\ldots cpp\ldots p$ and
    $pp\ldots pcc\ldots c$, each with $\tilde{c}$ letters $c$ and $\tilde{p}$
    letters~$p$. However these sign patterns are canonical, see part~(1) of
    Remarks~\ref{remscanrig}, and their corresponding canonical orders of
    moduli are different, see Definition~\ref{defican},
    so $\Omega$ cannot be equal to each of them.}
  \end{rem}

\begin{defi}\label{defiuniv}
  {\rm A sign pattern is called {\em universal} if it is realizable with all
    compatible orders of moduli.}
  \end{defi}

There exist natural candidates for the role of universal sign patterns.

\begin{nota}
  {\rm (1) We set $P_{m,n}:=(x-1)^m(x+1)^n$, $m$, $n\in \mathbb{N}$. We denote
    by $\sigma_j(P_{m,n})$ the sign $+$, $-$ or $0$ of the coefficient of
    $x^j$ of the polynomial $P_{m,n}$. 
It is clear that the polynomial $P_{m,n}$ is
{\em self-reciprocal}, i.~e. $x^{m+n}P_{m,n}(1/x)=(-1)^mP_{m,n}(x)$. Hence the
sign changes and sign preservations in $\sigma (P_{m,n})$ are situated
symmetrically w.r.t. the middle of $\sigma (P_{m,n})$. 
As $P_{m,n}(-x)=(-1)^{m+n}P_{n,m}(x)$, it suffices to study the polynomials
$P_{m,n}$ for $m\leq n$.
\vspace{1mm}

(2) We denote by $\Sigma_{k_1,k_2,k_3,\ldots}$ a sign pattern beginning with
$k_1$ signs $+$ followed by $k_2$
signs $-$ followed by $k_3$ signs $+$ etc.
For a generalized sign pattern we use a similar notation including zeros at the
places between two numbers $k_i$ where zero coefficients occur, see part (1)
of Lemma~\ref{lmMB}. Thus the
generalized sign pattern $(+,+,0,-,-,-,+,0,-,-)$ will be denoted also by
  $\Sigma _{2,0,3,1,0,2}$.}
  \end{nota}

\begin{rem}
  {\rm For $d\geq 4$, a sign pattern $\Sigma_{k_1,k_2,\ldots k_s}$ is canonical
    if and only if $k_i\neq 2$ for $i=2$, $\ldots$, $s-1$. For $d=3$,
    not canonical are exactly $\Sigma_{1,2,1}$ and $\Sigma_{2,2}$.}
  \end{rem}

\begin{tm}\label{tmuniv}
  (1) When the polynomial $P_{m,n}$ has no vanishing coefficients, the sign
  pattern $\sigma (P_{m,n})$ is universal.
\vspace{1mm}

  (2) For $d=2m+1$, a sign pattern with $m$ sign changes and $m+1$ sign
  preservations is universal if and only if this is the sign pattern
  $\sigma (P_{m,m+1})=\Sigma _{2,2,\ldots ,2}$.
  \vspace{1mm}
  
  (3) For $m$ even, one has $\sigma (P_{m,m+2})=\Sigma _{2,\ldots,2,3,2,\ldots ,2}$
  ($m/2$ times $2$, one $3$ and again $m/2$ times $2$). For $m$ odd,
  $\sigma (P_{m,m+2})=\Sigma_{2,2,\ldots ,2,0,2,2,\ldots ,2}$
  is a generalized sign pattern 
  consisting of $(m+1)/2$ couples of equal signs (the signs of two consecutive
  couples being opposite), one $0$ and again $(m+1)/2$ couples of equal signs.
\end{tm}

We give comments on the theorem in the next section and its proof in
Section~\ref{secprtmuniv}. While the proof of part (1) is short and needs no
preliminary results, the proofs of parts (2) and (3) rely on 
Section~\ref{secPmn} in which we study the question when the polynomial $P_{m,n}$
can have zero coefficients.

\section{Comments\protect\label{seccomments}}

If a hyperbolic polynomial $Q$ has at least one vanishing coefficient, then
it is not true that $\ell _-=\tilde{p}$, but only that $\ell _-$ is equal to
the number of sign changes in the sequence of coefficients of the polynomial
$Q(-x)$. Example: for $Q:=x^2-1$, one has
$\ell_+=\ell_-=\tilde{c}=1\neq \tilde{p}=0$; see also Remark~\ref{rem0}.
Bearing this difference in mind,
one can define realizability of
couples (generalized sign pattern,
order of moduli) as this is done for sign patterns in
part (2) of Definition~\ref{deficompat},
and then define universal generalized sign patterns by analogy with
Definition~\ref{defiuniv}.  

For $d\leq 6$, the following two statements hold true
(this can be directly checked using the results
in \cite{KoSo} and \cite{GaKoTaMC}):
\vspace{1mm}

\begin{itemize}

\item All universal sign patterns are of the form $\sigma (P_{m,n})$ for
  some $m$ and $n$.
  \vspace{1mm}
  
\item When $\sigma (P_{m,n})$ is a generalized sign pattern, not
  a sign pattern, then there exists no universal sign pattern with
  $\tilde{c}=m$ and $\tilde{p}=n$. Hence for $d\leq 6$,
  no generalized sign pattern is
  universal, otherwise all sign patterns obtained from it by replacing its
  zeros by $+$ or $-$ should also be universal. (One can perturb the
  coefficients of $P_{m,n}$ to make them all non-zero without changing
  the order of moduli.)
  \end{itemize}

It would be interesting to prove or disprove that these two statements hold
true for any $d\in \mathbb{N}$. 

In the next section we deal with the situation when $\sigma (P_{m,n})$ is not
a sign pattern, but a generalized sign pattern, so part (1) of
Theorem~\ref{tmuniv} cannot be applied. For $n>m$ and as $m\rightarrow \infty$,
in asymptotically $1/4$ of the cases there is a vanishing coefficient in
$P_{m,n}$. These are the cases when $m$ and $n$ are odd and this is the middle
coefficient, see part~(1) of Proposition~\ref{prop4parts}. Zero coefficients
not in the middle also occur, but they seem to appear less and less frequently
as $m\rightarrow \infty$, see Propositions~\ref{prop4parts} 
and~\ref{propcoeffxk}, Remarks~\ref{remsmn1}
and Remarks~\ref{remscoeffxk}.

\section{Vanishing coefficients of the polynomials $P_{m,n}$
  \protect\label{secPmn}}

\subsection{Generalized sign patterns of the polynomials $P_{m,n}$}

We begin this subsection by a result about the vanishing coefficients of
the polynomials $P_{m,n}$, in particular for small values of~$m$.

\begin{prop}\label{prop4parts}
  (1) For $m$ and $n$ odd, the middle coefficient of the polynomial
  $P_{m,n}$ vanishes. For any $m$ and $n$, if the coefficient of $x^j$ in
  $P_{m,n}$ vanishes, then the coefficient of $x^{m+n-j}$ also vanishes.
  For $m=1$ and $n$ even, the polynomial $P_{m,n}$ has no vanishing coefficient.
  For $m=1$ and $n$ odd, the middle coefficient of $P_{m,n}$ is its only
  vanishing coefficient. For $m$ and $n$ even, the middle coefficient
  of $P_{m,n}$ does not vanish.

  (2) For $m=2$ and $n\geq 2$,
  the polynomial $P_{m,n}$ has vanishing coefficients exactly
  when $n=\nu ^2-2$, $\nu \in \mathbb{N}^*$, $\nu \geq 2$. These are the
  coefficients of $x^j$, $j=(\nu \pm 1)\nu /2$, and the generalized sign
  pattern $\sigma (P_{2,n})$ equals $\Sigma _{(\nu -1)\nu /2,0,\nu -1,0,(\nu -1)\nu /2}$.
  
  For $\nu ^2-2<n<(\nu +1)^2-2$, one has

  $$\left\{ \begin{array}{ll}
    \sigma (P_{2,n})=\Sigma _{[(n-\nu )/2]+1,~\nu +1,~[(n-\nu )/2]+1}&{\rm 
  for}~~n-\nu~~{\rm even~~and}\\ \\ 
    \sigma (P_{2,n})=\Sigma _{[(n-\nu )/2]+2,~\nu ,~[(n-\nu )/2]+2}&{\rm for}~~n-\nu~~
           {\rm odd.}\end{array}\right.$$

   (3) For $m=3$ and $n\geq 3$, the polynomial $P_{m,n}$ has vanishing
  coefficients exactly
  in the following two cases:

  (a) when $n$ is odd and the coefficient of $x^{(n+3)/2}$ vanishes
  (see part~(1));

  (b) when $n=(\mu ^2-7)/3$, where $\mu \in \mathbb{N}$, $\mu \geq 4$, is not
  divisible by $3$; in this case vanishing are the coefficients of $x^j$,
  $j=(\mu \pm 1)(\mu \pm 2)/6$.

  Set $\gamma :=(\mu -1)(\mu -2)/6$. For $n=(\mu ^2-7)/3$, one obtains

  $$\sigma (P_{3,n})=\left\{ \begin{array}{ll}
    \Sigma _{\gamma ,0,(\mu -1)/2,(\mu -1)/2,0,\gamma}&
           {\rm for~}\mu \equiv 1~{\rm or~}5~\mod 6~{\rm and}\\ \\
           \Sigma_{\gamma ,0,\mu /2-1,0,\mu /2-1,0,\gamma}&{\rm for~}
           \mu \equiv 2~{\rm or~}4~\mod 6~.\end{array}\right.$$

  For $(\mu ^2 -7)/3<n<((\mu +1)^2-7)/3$,
  one has

$$\sigma (P_{3,n})=\left\{ \begin{array}{ll}
    \Sigma _{\rho ,\kappa ,0,\kappa ,\rho}&{\rm for~}n~{\rm odd,~with}\\ \\
    \rho =\frac{n+3-\mu}{2},~\kappa =\frac{\mu}{2}~{\rm for~}\mu{\rm ~even,}&
    \rho =\frac{n+4-\mu}{2},~
    \kappa =\frac{\mu -1}{2}~{\rm for~}\mu{\rm ~odd,}\\ \\ 
    \Sigma _{\rho ,\kappa ,\kappa ,\rho}&{\rm for~}n{\rm ~even,~with}\\ \\
    \rho =\frac{n+4-\mu}{2},~\kappa =\frac{\mu}{2}~{\rm for~}\mu{\rm ~even,}&
    \rho =\frac{n+3-\mu}{2},~
    \kappa =\frac{\mu +1}{2}~{\rm for~}\mu{\rm ~odd.}\end{array}\right.$$
    
  (4) For $m\in \mathbb{N}^*$, the coefficients of all odd powers of
  $x$ in $P_{m,m}$ and only they are~$0$ while $P_{m,m+1}$ and, for even
  $m$, $P_{m,m+2}$ have no vanishing coefficients. The only vanishing
  coefficient of $P_{m,m+2}$ for $m$ odd is the one of $x^{m+1}$.
\end{prop}

\begin{rem}
  {\rm The values of $n$ as defined in part (3)(b) of
    Proposition~\ref{prop4parts} form the following sequence
    $\mathcal{S}$ (see \cite[A369304]{OEIS}):~$3$, $6$, $14$,
    $19$, $31$, $38$, $54$, $63$, $83$, $94$, $118$, $131$, $\ldots$, in which 
    pairs of even numbers alternate with pairs of odd
    numbers. For an even $n$ of this sequence, the polynomial $P_{3,n}$
    has exactly two and for
    an odd $n$, it has exactly three zero coefficients, see
    part (3) of Proposition~\ref{prop4parts}. Consider the sequence
    of differences between two consecutive terms of the sequence
    $\mathcal{S}$: $3$, $8$, $5$, $12$, $7$, $16$, $9$, $20$, $11$, $24$, $13$,
    $\ldots$. It takes alternatively the values of the two arithmetic
    progressions $3$, $5$, $7$, $\ldots$ and $8$, $12$, $16$, $\ldots$.}
  \end{rem}

\begin{proof}[Proof of Proposition~\ref{prop4parts}]
  The first two claims of part (1) follow from the self-reciprocity of $P_{m,n}$.
  In particular, when $m$ and $n$ are odd, one has
  $x^{m+n}P_{m,n}(1/x)=-P_{m,n}(x)$ and
  the middle coefficient of $P_{m,n}$ is equal to its opposite, so it is~$0$.
  For $m=1$, one has $P_{m,n}=\sum _{j=0}^n({n\choose j-1}-{n\choose j})x^j$, so
  $P_{m,n}$ has a zero coefficient if and only if $n$ is odd and $j=(n+1)/2$.
  If $m$ and $n$ are even and the middle coefficient vanishes, then by
  part~(2) of Lemma~\ref{lmMB} there is a sign change in the middle of the sign
  pattern $\sigma (P_{m,n})$. By self-reciprocity the remaining sign changes
  are symmetrically situated w.r.t. the middle of $\sigma (P_{m,n})$, so
  the total number of sign changes is odd which contradicts the parity of~$m$.
\vspace{1mm}

  Part (2). For $m=2$, one has

  $$\begin{array}{cccccc}
    P_{2,n}&=&(x^2-2x+1)\sum _{j=0}^n{n\choose j}x^j&=&
  \sum _{j=0}^{n+2}b_{n,j}x^j~,&{\rm where}\\ \\ 
  b_{n,j}&:=&{n\choose j}-2{n\choose j-1}+{n\choose j-2}~.&&&\end{array}$$
  The condition $b_{n,j}=0$ reads

  $$\begin{array}{lccc}
    (n-j+2)(n-j+1)-2(n-j+2)j+j(j-1)&=&0~,&{\rm i.~e.}\\ \\ 
  4j^2-(4n+8)j+(n+2)(n+1)&=&0~,&\end{array}$$
  so one obtains $j=\tilde{j}_{\pm}:=(n+2\pm \sqrt{n+2})/2$.
  These values of $j$ are integer
  exactly when $n=\nu ^2-2$, $\nu \in \mathbb{N}^*$, $\nu \geq 2$, 
  from which the first two claims of part (2) follow.

  Suppose that $\nu ^2-2<n<(\nu +1)^2-2$, i.~e.

  $$\nu ^2<n+2<(\nu +1)^2~~~\, {\rm and}~~~\,
  (n+1-\nu )/2<\tilde{j}_-<(n+2-\nu )/2~.$$
  If $n-\nu$ is even, then

  $$[(n-\nu )/2]<[(n-\nu )/2]+1/2<\tilde{j}_-<[(n-\nu )/2]+1$$
  (where $[.]$ stands
  for the integer part). If $n-\nu$ is odd, then

  $$[(n-\nu )/2]+1<\tilde{j}_-<[(n-\nu )/2]+3/2<[(n-\nu )/2]+2$$
  from which the rest of
  part (2) follows. 
\vspace{1mm}

  Part (3). For $m=3$, one has

  $$\begin{array}{l}P_{3,n}=(x^3-3x^2+3x-1)\sum _{j=0}^n{n\choose j}x^j=
  \sum _{j=0}^{n+2}c_{n,j}x^j~,~~~\, {\rm where}\\ \\ 
  c_{n,j}:={n\choose j}-3{n\choose j-1}+3{n\choose j-2}-{n\choose j-3}~.
  \end{array}$$
  The condition $c_{n,j}=0$ yields

  $$\begin{array}{l}(n-j+3)(n-j+2)(n-j+1)-3(n-j+3)(n-j+2)j\\ \\
    +3(n-j+3)j(j-1)-
    j(j-1)(j-2)=0~,~~~\, {\rm i.~e.}\\ \\
2(j-(n+3)/2)(4j^2-(4n+12)j+(n+2)(n+1))=0~.
  \end{array}$$
  The first factor defines an integer solution $j=(n+3)/2$ exactly when $n$ is
  odd. This proves part (3)(a).

  The second factor is $0$ for
  $j=j^*_{\pm}:=(n+3\pm \sqrt{3n+7})/2$
  and $j^*_{\pm}$ are
  integer only if $n=(\mu ^2-7)/3$, $\mu \in \mathbb{N}$, in which case
  $n$ and $\mu$ are of opposite parity and
  $j^*_{\pm}=(\mu \pm 1)(\mu \pm 2)/6$. Both these values are integer exactly
  when $\mu$ is not divisible by $3$. This proves the first displayed formula
  of part (3)(b).

  It is clear that if $j^*_{\pm}$ are
  not integer, then $\rho =[j^*_-]+1$. The condition

  $$(\mu ^2 -7)/3<n<((\mu +1)^2-7)/3$$
  can be given the equivalent form

  $$-\mu -1<-\sqrt{3n+7}<-\mu ~,$$
  therefore

  $$(n+2-\mu )/2<j^*_-<(n+3-\mu )/2~,~~~\, {\rm i.~e.}$$

  $$\begin{array}{cl}
    (n-\mu )/2+1<j^*_-<(n-\mu )/2+3/2<(n-\mu +4)/2&
    {\rm for}~~n-\mu ~~{\rm even}~,\\ \\
    (n-\mu -1)/2+3/2<j^*_-<(n-\mu +3)/2&{\rm for}~~n-\mu ~~{\rm odd}~.
    \end{array}$$
  When $n$ and $\mu$ are of the same parity, then $n-\mu$ is even and
  $\rho =(n-\mu +4)/2$. When they are of different parity, then
  $n-\mu$ is odd and $\rho =(n-\mu +3)/2$.

  It remains to find~$\kappa$.
  To this end one has to notice that when $n$ is odd, then the length
  of the sign pattern is odd and $\kappa +\rho =(n+3)/2$, so
  $\kappa =\mu /2$ for $\mu$ even and $\kappa =(\mu -1)/2$ for $\mu$ odd.
  When $n$ is even, then the length
  of the sign pattern is even and $\kappa +\rho =(n+4)/2$, thus
  $\kappa =\mu /2$ for $\mu$ even and $\kappa =(\mu +1)/2$ for $\mu$ odd.
  This proves the second displayed formula of part (3)(b). 
  \vspace{1mm}

  Part (4). The first claim of part (4) results from $P_{m,m}=(x^2-1)^m$. As

  \begin{equation}\label{equP}
    P_{m,m+1}=(x+1)P_{m,m}=x\sum _{j=0}^m{m\choose j}(-1)^jx^{2j}+
  \sum _{j=0}^m{m\choose j}(-1)^jx^{2j}~,\end{equation}
  any coefficient of $P_{m,m+1}$ up to a sign is a binomial coefficient hence
  non-zero. As in equality (\ref{equP}) one finds that any coefficient of
  
  $$\begin{array}{ccl}P_{m,m+2}&=&(x^2+2x+1)P_{m,m}\\ \\
    &=&2x\sum_{j=0}^m(-1)^j{m\choose j}x^{2j}+
    \sum_{j=0}^{m+1}\left( (-1)^{j-1}{m\choose j-1}+
    (-1)^j{m\choose j}\right) x^{2j}\end{array}$$
  is either of the form $\pm 2{m\choose j}\neq 0$
  or of the form $\pm ({m\choose j}-{m\choose j-1})$. If $m$ is even, then
  the latter difference is non-zero for any $j$, $1\leq j\leq m$. If $m$ is
  odd, then this difference vanishes exactly when $j=(m+1)/2$ which gives the
  coefficient of~$x^{m+1}$.
\end{proof}

\begin{rems}\label{remsmn1}
  {\rm (1) It would be interesting to show that for any $m\geq 1$ fixed, there
    are infinitely many values of $n$ for which no coefficient of $P_{m,n}$
    vanishes and infinitely many values for which at least one coefficient
    is~$0$.
    \vspace{1mm}

    (2) In the proofs of parts (2) and (3) of Proposition~\ref{prop4parts},
    for $m=2$ and $3$, we obtained formulas for the values of $j$ expressed as
    functions of $n$ such that when $j\in \mathbb{N}$, then the coefficient of
    $x^j$ in $P_{m,n}$ vanishes. Part (1) of the proposition implies that for
    $m=1$, one obtains $j=(n+1)/2$. Similar formulas exist also
    for $m=4$ and $m=5$. For $m=4$, one obtains four values of $j$:

    $$(n+4\pm \sqrt{\pm \theta +3n+8})/2~,~~~\,
    \theta :=\sqrt{6n^2+30n+40}~.$$
    For $n=13$ and for $-\theta$, one obtains $j=(17\pm 3)/2$,
    i.~e. $j=7$ and~$10$; for $n=13$ and for $+\theta$, one gets $j=3.89\ldots$
    and $13.10\ldots$, i.~e. one does not obtain
    integer values for~$j$. Thus $\sigma (P_{4,13})=\Sigma_{4,3,0,2,0,3,4}$.
    For $n=62$ and for $-\theta$, one obtains
    $j=(66\pm 6)/2$, i.~e. $j=30$ and~$36$; for $n=62$ and for $+\theta$,
    one does not obtain integer values for~$j$ and in a similar way one
    concludes that 
    $\sigma (P_{4,62})=\Sigma_{24,6,0,5,0,6,24}$. For $m=4$ and
    $5\leq n\leq 200$, these are the only cases in which one 
    obtains vanishing coefficients in~$P_{m,n}$. For $n=151$, one has
    $\theta =2^347$, but neither of the two quantities $\pm \theta +3n+8$
    is an exact square.
    \vspace{1mm}
    
 For $m=5$, the five values of $j$ are $(n+5)/2$ and}

  $$(n+5\pm \sqrt{\pm \chi +5n+15})/2~,~~~\, \chi :=\sqrt{10n^2+50n+76}~.$$
  {\rm For $n=12$ and for $+\chi$, one obtains $j=(17\pm 11)/2$, i.~e.
  $j=3$ and~$14$; for $n=12$ and for $-\chi$, one does not obtain
    integer values for~$j$; $(n+5)/2$ is not integer for $n=12$. One has 
    $\sigma (P_{5,12})=\Sigma_{3,0,2,3,3,2,0,3}$. For $n=31$, one obtains
    $\chi =106$; for $-\chi$, one gets $j=(36\pm 8)/2$, i.~e. $j=14$ and~$22$;
    for $+\chi$, one does not obtain integer values for~$j$, so
    $\sigma (P_{5,31})=\Sigma_{10,4,0,3,0,3,0,4,10}$.

     Set $\tilde{\chi}_{\pm}:=\pm \chi +5n+15$.
     For $n=62$, one obtains $\chi =2^2\times 3\times 17=204$,
     $\tilde{\chi}_+=23^2$, $\tilde{\chi}_-=11^2$ and

    $$\begin{array}{ccccc}
      (62+5\pm \sqrt{\tilde{\chi}_+})/2&=&(67\pm 23)/2&=&45~~~\, 
    {\rm and}~~~\, 22~,\\ \\ 
    (62+5\pm \sqrt{\tilde{\chi}_-})/2&=&(67\pm 11)/2&=&39~~~\, {\rm and}~~~\, 
    28~,\end{array}$$
    so $\sigma (P_{5,62})=\Sigma_{22,0,5,0,5,5,0,5,0,22}$. 
    These are all cases with
    $6\leq n\leq 200$ for which there is more than
    one vanishing coefficient in~$P_{5,n}$.}
  \end{rems}

\subsection{Which triples of consecutive signs
  are possible for $\sigma (P_{m,n})$?}

The following lemma (see \cite[Lemma~7]{KoMB})
exhibits some constraints about the presence of vanishing
coefficients in hyperbolic polynomials.

\begin{lm}\label{lmMB}
  Suppose that $V$ is a hyperbolic polynomial of degree $d\geq 2$ with no
  root at $0$. Then:
  
  (1) $V$ does not have two or more consecutive vanishing coefficients.
  
  (2) If $V$ has a vanishing coefficient, then the signs of its surrounding
  two coefficients are opposite.
\end{lm}

\begin{rem}\label{rem0}
  {\rm For a hyperbolic degree $d$ polynomial $Q$ with non-zero constant term
    it is true that $\ell_+=\tilde{c}$. Indeed, denote by $\lambda$ the
    number of vanishing coefficients of $Q$. By Lemma~\ref{lmMB} the triples of
    signs of coefficients with a vanishing coefficient in the middle are
    of the form $(+,0,-)$ or $(-,0,+)$. Denote by $\tilde{c}^*$ the number of
    sign changes $(+,-)$ or $(-,+)$ and by $\tilde{p}^*$ the number of sign
    preservations of the form $(+,+)$ or $(-,-)$ in $\sigma (Q)$. Then
    $\tilde{c}=\tilde{c}^*+\lambda$ and the number of sign changes in
    $\sigma (Q(-x))$ equals $\tilde{p}^*+\lambda$. But by
    Descartes' rule of signs}

  $$\begin{array}{cc}
    \ell_+\leq \tilde{c}=\tilde{c}^*+\lambda ~,&\ell_-\leq
    \tilde{p}^*+\lambda \\ \\ {\rm and}&
    \tilde{c}^*+\tilde{p}^*+2\lambda =d=\ell_++\ell_-~,\end{array}$$
  {\rm so $\ell_+=\tilde{c}^*+\lambda =\tilde{c}$ and
    $\ell_-=\tilde{p}^*+\lambda$. Observe that $\tilde{p}=\tilde{p}^*$.
    Hence if $\lambda >0$, then $\ell_->\tilde{p}$.}
  \end{rem}

With the help of Lemma~\ref{lmMB} we prove another proposition about the
polynomials $P_{m,n}$:

\begin{prop}\label{propPmn}
  (1) For $m<n$, if a component of $\sigma (P_{m,n})$ equals $+$ (resp. equals
  $-$),
  then at least one of the two surrounding components equals $+$ (resp. $-$).
  In particular, $\sigma _{m+n-1}(P_{m,n})=+$ and
  $\sigma _1(P_{m,n})=\sigma _0(P_{m,n})={\rm sign}((-1)^m)$.

  (2) There are $20$ triples of consecutive components of a generalized sign
  pattern compatible with Lemma~\ref{lmMB}. Out of these the following $12$
  can be encountered in some polynomial $P_{m,n}$:

  \begin{equation}\label{equpossible}
    (\pm ,\pm ,\pm )~,~~~\, (\pm ,\mp ,\mp )~,~~~\, (\pm ,\pm ,\mp )~,~~~\,
  (0,\pm ,\pm )~,~~~\,
  (\pm ,\pm ,0)~~~\, {\rm and}~~~\, (\pm ,0,\mp )~,\end{equation}
while the following $8$ cannot be encountered:

\begin{equation}\label{equimpossible}
  (\pm ,\mp ,\pm )~,~~~\, (0,\pm ,0)~,~~~\, (0,\pm ,\mp )~~~\,
{\rm and}~~~\, (\pm ,\mp ,0)~.\end{equation}
\end{prop}

\begin{cor}
  For $1\leq m<n$, the (generalized) sign pattern $\sigma (P_{m,n})$ contains 
  strings of signs $+$ and strings of signs $-$. All these strings are
  of length $\geq 2$, where
  two consecutive strings consist of opposite signs and they 
  either follow one another immediately or are separated by one~$0$. Hence
  in the (generalized) sign patterns $\sigma (P_{m,n})$ all
  sign changes are isolated.
  \end{cor}

\begin{proof}[Proof of Proposition~\ref{propPmn}]
  There are $27$ a priori possible triples of consecutive components of
  generalized sign patterns -- each  
  component equals $+$, $-$ or $0$. The 7 triples $(0,0,0)$, $(\pm ,0,0)$,
  $(0,0,\pm )$ and $(\pm ,0, \pm )$, and only they, are not compatible
  with Lemma~\ref{lmMB}.
  This proves the first claim of part~(2).

   As $P_{2,6}=x^8+4x^7+4x^6-4x^5-10x^4-4x^3+4x^2+4x+1$, one encounters the triples
  $(\pm ,\pm ,\pm )$, $(\pm ,\mp ,\mp )$ and $(\pm ,\pm ,\mp )$ in
  $\sigma (P_{2,6})$ and the remaining $6$ triples of (\ref{equpossible}) in
  $\sigma (P_{2,14})$, where

  $$\sigma (P_{2,14})=(+,+,+,+,+,+,0,-,-,-,0,+,+,+,+,+,+)~.$$ 
  This proves the second claim of part~(2).

  The rest of the proposition is proved by induction on $m$. The induction base
  is the case of $P_{1,n}$, $n\in \mathbb{N}$, $n>1$, with
  $\sigma (P_{1,n})=(+,\ldots ,+,0,-,\ldots ,-)$ for $n$ odd ($(n+1)/2$ times
  $+$ and $(n+1)/2$ times $-$) and $\sigma (P_{1,n})=\Sigma _{n/2+1,n/2+1}$ for
  $n$ even; in this case all claims of the proposition are readily checked.

   Suppose that the proposition holds true for $P_{m,n}$. Consider
  the polynomial $P_{m+1,n+1}=(x^2-1)P_{m,n}$. Suppose that $\sigma (P_{m,n})$
  contains the following sign change:

  $$\begin{array}{cccc}
    {\rm either}&(\sigma _k(P_{m,n}),\ldots ,\sigma _{k-3}(P_{m,n}))&=&
    (\pm ,\pm ,\mp ,\mp )\\ \\ {\rm or}&
    (\sigma _k(P_{m,n}),\ldots ,\sigma _{k-4}(P_{m,n}))&=&
  (\pm ,\pm ,0,\mp ,\mp )~.\\ \\ 
    {\rm Then}&(\sigma _k(P_{m+1,n+1}),\sigma _{k-1}(P_{m+1,n+1}))&=&
    (\mp ,\mp )\\ \\ 
  {\rm or}&(\sigma _k(P_{m+1,n+1}),\ldots ,\sigma _{k-2}(P_{m+1,n+1}))&=&
  (\mp ,\mp ,\mp )\end{array}$$
  respectively. Thus in $\sigma (P_{m+1,n+1})$ there are $m$
  strings of consecutive equal signs $+$ or $-$ 
  corresponding
  to the $m$ sign changes in $\sigma (P_{m,n})$. It is clear that

  $$\begin{array}{cccccc}\sigma _{m+n+2}(P_{m+1,n+1})&=&\sigma_{m+n+1}(P_{m+1,n+1})&=&+
    &
    {\rm and}\\ \\  
    \sigma _1(P_{m+1,n+1})&=&\sigma_0(P_{m+1,n+1})&=&{\rm sgn}((-1)^{m+1})
  \end{array}$$
  which defines another $2$ strings. The signs $\pm$ of these
  $m+2$ strings alternate; to prove this one can use the fact that
  $\sigma _k(P_{m,n})=-\sigma_k(P_{m+1,n+1})$.
  Hence there are $\geq m+1$ sign changes in
  $\sigma(P_{m+1,n+1})$.

But by Descartes' rule of signs these changes are
  exactly $m+1$, see Remark~\ref{rem0}. We show that then $\sigma(P_{m+1,n+1})$
  cannot contain a triple of
  consecutive components $(\pm ,\mp ,\pm )$, $(0,\pm ,\mp )$, 
  $(\pm ,\mp ,0)$ or $(0,\pm ,0)$.
  Indeed, consider two consecutive strings of equal signs
  in $\sigma (P_{m+1,n+1})$ corresponding to sign changes in $\sigma (P_{m,n})$:
  $\omega_L:=(\pm ,\ldots ,\pm )$ and
  $\omega _R:=(\mp ,\ldots ,\mp )$. We consider the case when
  $\omega_L:=(+,\ldots ,+)$ and $\omega _R:=(-,\ldots ,-)$, the opposite
  case is treated by analogy. 
  We assume that $\omega_L$ and $\omega _R$ are maximal, i.~e.
  they cannot be further extended.

  Denote by $\omega_C$ the possible string of components of
  $\sigma (P_{m+1,n+1})$
  between them. Suppose that $\omega _C$ has exactly two components. Then
  they cannot be $(0,0)$, $(0,+)$ or $(-,0)$ (contradiction with
  Lemma~\ref{lmMB}), neither $(+,0)$, $(0,-)$, $(+,+)$, $(+,-)$ or $(-,-)$
  (contradiction with the maximality of $\omega _L$ and $\omega _R$). If they
  are $(-,+)$, then there will be more that $m+1$ sign changes in
  $\sigma (P_{m+1,n+1})$. Hence $\omega_C$ cannot have exactly two components. 

Suppose that $\omega _C$ has $w\geq 3$ components. Set
  $\omega _C:=(\alpha _1,\ldots ,\alpha _w)$. 
  Then in the same way $\alpha _1=0$ or $\alpha _1=-$ while
  $\alpha _w=0$ or $\alpha _w=+$. If $\alpha _1=0$
  (resp. $\alpha_w =0$), then by Lemma~\ref{lmMB},
  $\alpha_2=-$ (resp. $\alpha_{w-1}=+$). If $(\alpha_1,\alpha_2)=(0,-)$, then
  between $\alpha _2$ and $\omega_R$ there is no $+$, otherwise
  $\sigma (P_{m+1,n+1})$ has more than $m+1$ sign changes. If all components
  between $\alpha _2$ and $\omega_R$ equal $-$,
  then this contradicts the maximality of $\omega_C$. If there is at least
  one $0$, this contradicts Lemma~\ref{lmMB}.
  So $(\alpha_1,\alpha_2)=(0,-)$ is impossible and similarly one cannot have
  $(\alpha_{w-1},\alpha_w)=(+,0)$. But then $\alpha_1=-$ and $\alpha_w=+$ which
  means that $\sigma (P_{m+1,n+1})$ has more than $m+1$ sign changes.

 Thus either $\omega =\emptyset$ or $\omega$ consists of just one component.
  As $\omega_L$ and $\omega_R$ are maximal, this component equals~$0$.
  Thus in the string $\omega_L\omega_C\omega_R$
  (and hence in $\sigma (P_{m+1,n+1})$) one does not encounter any
  of the triples (\ref{equimpossible}). 
\end{proof}

\subsection{When does the coefficient of $x^k$ of $P_{m,n}$ vanish?}

In this subsection we ask the question when for $k=1$, $\ldots$, $5$ does
the coefficient of $x^k$ of the polynomial $P_{m,n}$ vanish.

\begin{rem}
  {\rm The coefficient of $x$ of the polynomial $P_{m,n}$
  vanishes if and
  only if $m=n$. Indeed, this coefficient equals ${n\choose 1}-{m\choose 1}$.
  Hence this is the (even) polynomial $(x^2-1)^m$.}
  \end{rem}

\begin{prop}\label{propcoeffxk}
  (1) For $m<n$, the coefficient of $x^2$ of the polynomial $P_{m,n}$
  vanishes if and
  only if $m=s(s-1)/2$, $n=s(s+1)/2$,
  $s\in \mathbb{N}^*$.

  (2) For $m<n$, the coefficient of $x^3$ of the polynomial $P_{m,n}$
  vanishes if and only if $m=(s-1)(s-2)/6$,
  $n=(s+1)(s+2)/6$, where $s\in \mathbb{N}^*$
  is not divisible by~$3$.
\end{prop}

Hence there exist infinitely many values of $m$ for which there exists $n>m$
such that the polynomial $P_{m,n}$ has more than one vanishing coefficient.

  \begin{proof}
    Part (1). The coefficient of $x^2$ of $P_{m,n}$ equals

    $$m(m-1)/2-mn+n(n-1)/2=((n-m)^2-(m+n))/2~.$$
    It vanishes exactly when $(n-m)^2=m+n$. Set $n-m=:s$, $m+n=:d$.
    Hence $d=s^2$, $m=s(s-1)/2$,
    $n=s(s+1)/2$, $s\in \mathbb{N}^*$.

     Part (2). The coefficient of $x^3$ of $P_{m,n}$ equals

    $${n\choose 3}-{n\choose 2}m+n{m\choose 2}-{m\choose 3}=
    (n-m)((n-m)^2-3(n+m)+2)/6~.$$
    It vanishes exactly when $d:=n+m=((n-m)^2+2)/3$, so one sets
    $s:=n-m$ hence $n+m=(s^2+2)/3$, $n=(s+1)(s+2)/6$
    and $m=(s-1)(s-2)/6$. For $s\in \mathbb{N}^*$, the
    numbers $m$ and $n$ are integer if and only if $s$
    is not divisible by~3. 
  \end{proof}

  \begin{rems}\label{remscoeffxk}
    {\rm (1) One can try to find in the same way values of the couple
      $(m,n)$ for
      which the coefficient of $x^4$ of the polynomial $P_{m,n}$ vanishes. This
      coefficient equals $\kappa :=(s^4-6s^2d+8s^2+3d^2-6d)/24$,
      where $d=m+n$ and $s=n-m$. The condition $\kappa =0$ yields}

    \begin{equation}\label{equd1}
      d=s^2+1\pm \sqrt{1+2(s^4-s^2)/3}~.
      \end{equation}
    {\rm For $s\leq 20$, one gets integer values for $m=(d-s)/2$ and $n=(d+s)/2$
      only for $s=1$, $2$, $3$ and $6$. One can exclude the trivial cases
      with $m+n<4$. The remaining cases are:}

$$(m,n)~=~(3,5),~(7,10),~(30,36)~~~\, {\rm and}~~~\, (1,7)~.$$
    {\rm The latter couple is the only one in which one chooses the
      sign ``$-$'' in
      formula~(\ref{equd1}).
   We list the signs of the last $6$ coefficients of
   the polynomials $P_{m,n}$ corresponding to these $4$ cases:}

    $$(-,0,+,+,-,-),~(-,0,+,+,-,-),~(+,0,-,-,+,+)~~~\, {\rm and}~~~\,
    (+,0,-,-,-,-)~.$$
    {\rm For $21\leq s\leq 200$, only for $s=91$ does one find integer couples
    $(m,n)$, namely $(715,806)$ with the sign ``$-$'' and $(7476,7567)$
      with the sign~``$+$''. The signs of their $6$ last coefficients are
    $(+,0,-,-,-,-)$ and $(+,0,-,-,+,+)$ respectively.}
    \vspace{1mm}
    
    {\rm (2) Consider the couples $(m,n)$ for which the coefficient of $x^5$
      of $P_{m,n}$ equals~0. With the notation of part~(1) this coefficient
      is equal to $d(s^4-10s^2d+20s^2+15d^2-50d+24)/120$. It vanishes for}

    \begin{equation}\label{equd2}
d=(1/15)(5s^2+25\pm \sqrt{10s^4-50s^2+265})~.
    \end{equation}
        {\rm For $s\leq 20$, integer values for $m$ and $n$ are obtained for
          $s=1$, $2$, $3$, $4$, $8$ and $11$. Eliminating the couples $(m,n)$
          with $m+n<5$ leaves $5$ couples:}

        $$(m,n)~=~(3,7),~(14,22),~(1,9),~(28,39)~~~\, {\rm and}~~~\, (3,14)~.$$
        {\rm The signs of their last $7$ coefficients equal respectively}

        $$\begin{array}{lcl}
          (-,0,+,+,-,-,-)~,&(+,0,-,-,+,+,+)~,&(+,0,-,-,-,-,-)~,\\ \\
          (+,0,-,-,+,+,+)&{\rm and}&(+,0,-,-,-,-,-)~.
        \end{array}$$
        {\rm For the third and fifth case one chooses the sign ``$-$'' (see
          (\ref{equd2})), for the rest the sign~``$+$''. For $21\leq s\leq 200$,
          one obtains integer couples $(m,n)$ only for $s=23$. These are 
          $(22,45)$ with the sign ``$-$'' and $(133,156)$
          with the sign~``$+$''. The signs of their $7$ last coefficients are
        $(-,0,+,+,+,+,+)$ and $(-,0,+,+,-,-,-)$ respectively.}
    \end{rems}

\section{Proof of Theorem~\protect\ref{tmuniv} 
  \protect\label{secprtmuniv}}

Part (1). For any order $\Omega$ of $m$ moduli of positive and $n$
  moduli of negative roots,
  one can find $m$ positive numbers $\alpha _i$ and $n$ negative numbers
  $-\beta_j$ such that their moduli are all distinct and
  realize the order~$\Omega$. Then for all
  $\varepsilon >0$ sufficiently small, the moduli of the numbers
  $1+\varepsilon \alpha _i$ and $-1-\varepsilon \beta _j$ are also distinct
  and realize the order $\Omega$.

  On the other hand for $\varepsilon >0$ sufficiently small, the
  coefficients of the polynomial

  $$\tilde{P}:=\prod_{i=1}^m(x-1-\varepsilon \alpha _i)\prod_{j=1}^n(x+1+
  \varepsilon \beta _j)$$
  are close to the ones of the polynomial $P_{m,n}$,
  so $\sigma (\tilde{P})=\sigma (P_{m,n})$. Hence $\sigma (P_{m,n})$ is a
  universal sign pattern.
  \vspace{1mm}
  
Part (2). A universal sign pattern as in the theorem is realizable
  with the
  order of moduli $NPNPN\ldots N$. This is a rigid order, see part (2) of
  Remarks~\ref{remscanrig}. It is realizable
  only with the sign pattern $\Sigma _{2,2,\ldots ,2}$. On the other hand
  by part (4) of Proposition~\ref{prop4parts} the polynomial $P_{m,m+1}$
  has no vanishing coefficients. This polynomial has $2m+2$ coefficients and
  $m$ sign changes.
  By part (1) of Proposition~\ref{propPmn}, $\sigma (P_{m,m+1})$ contains
  $m+1$ sequences of equal signs, each
  of length $\geq 2$, so 
  $\sigma (P_{m,m+1})=\Sigma _{2,2,\ldots ,2}$.
  \vspace{1mm}
  
  Part (3). Suppose that $m$ is even. The polynomial $P_{m,m+2}$ has $2m+3$
  coefficients. By part (1) of Proposition~\ref{propPmn}, there must be
  $m+1$ sequences of at least two consecutive coefficients of equal signs.
  Hence there are
  $m$ sequences of two and one sequence of three such coefficients.
  The polynomial is self-reciprocal, so the sequence of three coefficients
  is in the middle.

  Suppose that $m$ is odd. The middle coefficient of $P_{m,m+2}$ is $0$.
  By part (1) of Proposition~\ref{propPmn}, the rest
  of its $2m+3$ coefficients form $m+1$ couples of two consecutive
  coefficients of equal signs, $(m+1)/2$ to the left and $(m+1)/2$
  to the right of the~$0$ in the middle.

\end{document}